\documentclass[11pt]{amsart}

\usepackage[margin=1in]{geometry}
\usepackage{amsmath,amssymb,amsthm,amsfonts,mathrsfs}
\usepackage{enumitem}
\usepackage{tikz-cd}
\usepackage{hyperref}

\hypersetup{
    colorlinks=true,
    linkcolor=blue!60!black,
    citecolor=blue!60!black,
    urlcolor=blue!60!black
}

\newtheorem{theorem}{Theorem}[section]
\newtheorem{proposition}[theorem]{Proposition}
\newtheorem{lemma}[theorem]{Lemma}
\newtheorem{corollary}[theorem]{Corollary}

\theoremstyle{definition}
\newtheorem{definition}[theorem]{Definition}

\newtheorem{remark}[theorem]{Remark}
\newtheorem{question}[theorem]{Question}

\DeclareMathOperator{\supp}{supp}
\DeclareMathOperator{\acl}{acl}

\newcommand{\ZF}{\mathsf{ZF}}
\newcommand{\ZFA}{\mathsf{ZFA}}
\newcommand{\ZFC}{\mathsf{ZFC}}
\newcommand{\AC}{\mathsf{AC}}
\newcommand{\Fin}{\mathrm{Fin}}
\newcommand{\DFin}{\mathrm{D\text{-}Fin}}
\newcommand{\BFin}{\mathrm{B\text{-}Fin}}
\newcommand{\CFin}{\mathrm{C\text{-}Fin}}
\newcommand{\EFin}{\mathrm{E\text{-}Fin}}
\newcommand{\AFin}{\mathrm{A\text{-}Fin}}
\newcommand{\HFin}{\mathrm{H\text{-}Fin}}

\newcommand{\RnFin}[1]{R_{#1}\text{-}\mathrm{Fin}}
\newcommand{\FIN}{\mathrm{FIN}}
\newcommand{\FU}{\mathrm{FU}}
\newcommand{\TS}{\mathrm{TS}}
\newcommand{\RRT}{\mathrm{RRT}}
\newcommand{\CRT}{\mathrm{CRT}}
\newcommand{\RT}{\mathrm{RT}}
\newcommand{\NN}{\mathcal{N}}

\newcommand{\RRTFin}[2]{\RRT^{#1}_{#2}\text{-}\mathrm{Fin}}
\newcommand{\CRTFin}{\CRT\text{-}\mathrm{Fin}}

\title[Ramsey-theoretic finiteness in choiceless set theory]
{Ramsey-theoretic finiteness in choiceless set theory:\\ the Gowers collapse, rainbow and canonical separations}

\author{Brinda Venkataramani}
\address{Department of Mathematics, University of Toronto, 40 St.~George Street, Toronto, Ontario M5S 2E4, Canada}
\email{brinda.venkataramani@mail.utoronto.ca}

\date{July 23, 2026}

\begin{document}

\begin{abstract}
We study finiteness classes arising from Ramsey-theoretic principles in set theory without the Axiom of Choice.  First, we answer a question of Brot, Cao, and Fern\'andez-Bret\'on concerning Gowers' \(\FIN_k\) theorem.  The \(\FIN_k\) operation gives rise to a finiteness notion for each \(k\geq 1\), but we show that this notion is independent of the value of \(k\): the resulting hierarchy collapses, and every level is equivalent to \(H\)-finiteness, i.e., to Dedekind-finiteness of \([X]^{<\omega}\). Thus the Gowers classes do not produce new levels beyond \(H\)-finiteness.  We then turn to the Rainbow Ramsey theorem and the Canonical Ramsey theorem.  The corresponding failure classes \(\RRTFin{m}{n}\) and \(\CRTFin\) are genuine finiteness classes contained in \(\DFin\).  We prove several inclusions placing them relative to standard Ramsey-theoretic finiteness classes of Brot--Cao--Fern\'andez-Bret\'on, prove their independence from the collapsed Hindman--Gowers class, and analyze the parameter dependence of the rainbow classes through Fraenkel--Mostowski examples.  The finer questions of strictness, and the full two-parameter structure of the rainbow classes, are left open where the present arguments do not settle them.
\end{abstract}

\maketitle

\section{Introduction}

A \emph{finiteness class} is a class \(\mathcal F\) of sets satisfying:
\begin{enumerate}[label=(\roman*)]
    \item closure under subsets;
    \item closure under equipotence;
    \item every finite set belongs to \(\mathcal F\);
    \item \(\omega\notin \mathcal F\).
\end{enumerate}
Every finiteness class lies between \(\Fin\), the class of finite sets, and the largest finiteness class, which consists of all sets that do not have a countably infinite subset; such sets are called \emph{Dedekind-finite}, and their class is denoted \(\DFin\).  See Truss \cite{TrussFiniteness} for a systematic study of finiteness in this sense.  Under the Axiom of Choice, \(\Fin=\DFin\) and there is nothing further to say.  Without choice, however, the situation changes dramatically: many inequivalent finiteness classes arise, each measuring a different obstruction to building countable or order-like structure on a set.  We say that two finiteness classes are \emph{independent} if neither of them is provably contained in the other in \(\ZF\).

Brot, Cao, and Fern\'andez-Bret\'on \cite{BCF} studied finiteness classes arising from failures of Ramsey-theoretic statements in \(\ZF\).  In particular, they introduced finiteness notions associated with Ramsey's theorem and Hindman's finite unions theorem, and asked what analogous notions arise from Gowers' \(\FIN_k\) theorem.  The first goal of this paper is to answer that question.  The answer is that the apparent Gowers hierarchy collapses completely:
\[
    G_k\text{-}\Fin=H\text{-}\Fin
    \qquad\text{for every }k\geq 1.
\]
More precisely, \(X\) is \(G_k\)-finite if and only if \([X]^{<\omega}\) is Dedekind-finite, equivalently if and only if \(X\) is \(H\)-finite.  The property that \([X]^{<\omega}\) is Dedekind-finite admits many equivalent formulations; see \cite{BlackadarFarahKaragila}, \cite{KaragilaSchlicht}, and \cite{KeremedisWajch}, where sets with this property are called \emph{Cohen finite} or \emph{quasi Dedekind-finite}.  We shall use the term \(H\)-finiteness throughout the paper.

The collapse of the Gowers hierarchy motivates the second part of the paper.  If Gowers' theorem does not produce new finiteness classes beyond \(H\)-finiteness, then one should ask which Ramsey-theoretic principles do.  We study two such principles: the Rainbow Ramsey theorem (\(\RRT^m_n\)) and the Canonical Ramsey theorem (\(\CRT\)), as investigated in the choiceless setting by Banerjee, Gopaulsingh, and Moln\'ar \cite{BGM}; the statements \(\RRT^m_n(X)\) and \(\CRT(X)\) are defined precisely in Section~\ref{sec:rainbow}.  For arbitrary sets \(X\), we define \(X\) to be \(\RRT^m_n\)-finite if \(\RRT^m_n(X)\) fails, and \(\CRT\)-finite if \(\CRT(X)\) fails.  These give finiteness classes \(\RRTFin{m}{n}\) and \(\CRTFin\).

The resulting picture is different from the Gowers case.  The rainbow and canonical classes do not collapse to \(H\)-finiteness.  Indeed, \(\RRTFin{m}{n}\) is independent from \(\HFin\), and \(\CRTFin\) is independent from \(\HFin\).  Thus the Gowers collapse should be viewed as a rigidity phenomenon particular to the setting of finite unions and tetris sums, not as a general feature of Ramsey-theoretic finiteness.

\subsection*{Main results}

Some of the notation used below (in particular \(\RnFin{n}\), \(\RRTFin{m}{n}\), and \(\CRTFin\)) will be defined in the relevant sections; we state the results here for the reader's convenience.

The main results are as follows.

\begin{enumerate}[label=(\arabic*)]
    \item (Theorem~\ref{thm:gowers-collapse}.)  For every \(k\geq 1\), \(G_k\)-finiteness is equivalent to \(H\)-finiteness and to Dedekind-finiteness of \([X]^{<\omega}\).  Thus the Gowers hierarchy collapses to the \(H\)-finite class.

    \item (Corollaries~\ref{cor:rrt-fin-class} and~\ref{cor:crt-fin-class}.)  For every \(m,n\geq 2\), the classes \(\RRTFin{m}{n}\) and \(\CRTFin\) are finiteness classes.

    \item (Proposition~\ref{prop:bfin-crt} and Theorem~\ref{thm:crt-implies-rt}.)  The canonical Ramsey class contains several familiar choiceless finiteness classes.  In particular, \(\BFin\subseteq \CRTFin\), and \(\RnFin{n}\subseteq \CRTFin\) for every finite \(n\geq 2\).

    \item (Theorems~\ref{thm:rrt-hfin-indep} and~\ref{thm:crt-hfin-indep}.)  The new classes are independent from the collapsed Hindman--Gowers class: \(\RRTFin{m}{n}\not\subseteq \HFin\) and \(\HFin\not\subseteq \RRTFin{m}{n}\), and likewise \(\CRTFin\not\subseteq \HFin\) and \(\HFin\not\subseteq \CRTFin\).

    \item (Section~\ref{sec:parameter}.)  We give Fraenkel--Mostowski examples, including block models and Rado-type models, illustrating the parameter behavior of the rainbow classes.  Some of the stronger general separations are isolated as conjectural or conditional statements where the present argument still requires a more detailed orbit analysis.
\end{enumerate}

\section{Background finiteness classes}

We work in \(\ZF\), unless explicitly stated otherwise.  We use \(\omega\) for the set of natural numbers.

\begin{definition}
A set \(X\) is \emph{Dedekind-finite}, or \(D\)-finite, if there is no injection \(\omega\to X\).  We write \(\DFin\) for the class of Dedekind-finite sets.
\end{definition}

Equivalently, \(X\) is Dedekind-finite if no subset of \(X\) is equipotent with \(\omega\).  It is standard that \(\DFin\) is the largest finiteness class: if \(\mathcal F\) is a finiteness class and \(X\in\mathcal F\), then \(X\) must be Dedekind-finite.

We recall the following standard weak finiteness notions.

\begin{definition}
Let \(X\) be a set.
\begin{enumerate}[label=(\arabic*)]
    \item \(X\) is \emph{\(A\)-finite} if it cannot be partitioned into two disjoint infinite sets.
    \item \(X\) is \emph{\(B\)-finite} if no infinite subset of \(X\) is linearly orderable.
    \item \(X\) is \emph{\(C\)-finite} if there is no surjection \(X\twoheadrightarrow \omega\).
    \item \(X\) is \emph{\(E\)-finite} if no proper subset of \(X\) surjects onto \(X\).
\end{enumerate}
For any of these notions, \(X\) is called \(I\)-infinite if it is not \(I\)-finite.
\end{definition}

Infinite \(A\)-finite sets are commonly referred to as \emph{amorphous} sets.

The following implications are standard:
\[
    \AFin\subseteq\BFin,
    \qquad
    \AFin\subseteq\CFin,
    \qquad
    \CFin\subseteq\EFin.
\]
These are the only \(\ZF\)-provable inclusions among these four classes; see Truss \cite{TrussFiniteness} and \cite{JechAC,Truss} for the relevant independence results.

We also recall the Hindman-type class from \cite{BCF}.

\begin{definition}
Let \(X\) be a set and let \(Y\subseteq [X]^{<\omega}\).  Define
\[
    \FU(Y)=\left\{\bigcup_{i=1}^r y_i \;\middle|\; r\geq 1,\ y_1,\dots,y_r\in Y\text{ distinct}\right\}.
\]
If the elements of \(Y\) are pairwise disjoint, this is the usual finite-unions set generated by \(Y\).
\end{definition}

\begin{definition}
A set \(X\) is \emph{\(H\)-finite} if there is a coloring \(c:[X]^{<\omega}\to 2\) such that for every infinite pairwise disjoint \(Y\subseteq [X]^{<\omega}\), the set \(\FU(Y)\) is not monochromatic for \(c\).  The class of \(H\)-finite sets is denoted \(\HFin\).
\end{definition}

\begin{definition}
For \(n\geq 2\), let \(\RT^2_n(X)\) denote the assertion that every coloring \(c:[X]^2\to n\) admits an infinite homogeneous subset, i.e., Ramsey's theorem for pairs and \(n\) colors on \(X\).  A set \(X\) is \emph{\(R_n\)-finite} if \(\RT^2_n(X)\) fails, that is, if there is a coloring \(c:[X]^2\to n\) with no infinite homogeneous subset.  We denote the corresponding class by \(\RnFin{n}\).
\end{definition}

For the position of Ramsey's theorem among weak choice principles, see Blass \cite{Blass77}.

\section{The Gowers collapse}

We now define the finiteness notions arising from Gowers' \(\FIN_k\) theorem and prove that they all coincide with \(\HFin\).

\subsection{The spaces \(\FIN_k(X)\)}

\begin{definition}
Let \(X\) be a set and let \(k\geq 1\).  Define
\[
    \FIN_k(X)=
    \left\{
        f:X\to\{0,1,\dots,k\}
        \;\middle|\;
        f^{-1}(k)\neq\varnothing\text{ and }\supp(f)\text{ is finite}
    \right\},
\]
where \(\supp(f)=\{x\in X\mid f(x)\neq 0\}\).
\end{definition}

Thus \(\FIN_1(X)\) is naturally identified with \([X]^{<\omega}\setminus\{\varnothing\}\) by sending \(f\) to \(f^{-1}(1)=\supp(f)\).

\begin{definition}[Tetris operation]
For \(k\geq 2\), the \emph{tetris operation} is the map
\[
    T:\FIN_k(X)\to \FIN_{k-1}(X)
\]
defined by
\[
    T(f)(x)=\max(f(x)-1,0).
\]
Higher iterates are defined by \(T^0=\mathrm{id}\) and \(T^{j+1}=T\circ T^j\) as long as the expression is defined.  We also write \(T^k(f)=0\) for the zero function; this lies outside \(\FIN_j(X)\) for \(j\geq 1\), but it is convenient notation when describing sums.
\end{definition}

If \(f\in\FIN_k(X)\), then \(T(f)\in\FIN_{k-1}(X)\), since every point at level \(k\) is sent to level \(k-1\).  Moreover, for \(0\leq j\leq k-1\),
\[
    \supp(T^j(f))=f^{-1}\bigl(\{j+1,\dots,k\}\bigr).
\]

\begin{definition}[Block families]
A family \(Y\subseteq \FIN_k(X)\) is a \emph{block family} if its elements have pairwise disjoint supports: whenever \(y,y'\in Y\) are distinct,
\[
    \supp(y)\cap \supp(y')=\varnothing.
\]
\end{definition}

In the classical ordered setting \(X=\omega\), one usually works with block sequences \((x_i)_{i<\omega}\) satisfying \(\max\supp(x_i)<\min\supp(x_{i+1})\).  Since arbitrary sets need not be linearly orderable, we use block families instead.

\begin{definition}[Tetris sums]
Let \(Y\subseteq \FIN_k(X)\) be a block family.  The set of \emph{tetris sums} generated by \(Y\), denoted \(\TS(Y)\), consists of all functions of the form
\[
    \sum_{i=1}^r T^{t_i}(y_i),
\]
where \(r\geq 1\), \(y_1,
\dots,y_r\in Y\) are distinct, \(t_i\in\{0,1,\dots,k-1\}\), and at least one \(t_i=0\).  The sum is computed coordinatewise.
\end{definition}

Because the supports of the \(y_i\)'s are disjoint, the coordinatewise sum is simply the union of the individual functions on disjoint supports.  The requirement that some \(t_i=0\) ensures that the resulting function still attains the value \(k\), hence belongs to \(\FIN_k(X)\).

\begin{theorem}[Gowers' \(\FIN_k\) theorem \cite{Gowers}]\label{thm:gowers}
For every \(k\geq 1\), every finite coloring of \(\FIN_k(\omega)\) admits an infinite block sequence whose tetris sums are monochromatic.
\end{theorem}

For \(k=1\), this is Hindman's finite-unions theorem \cite{Hindman}: \(\FIN_1(\omega)\) is identified with the nonempty finite subsets of \(\omega\), and the tetris sums are exactly finite unions.

Since Theorem~\ref{thm:gowers} will be applied inside \(\ZF\) below, let us note that it is a theorem of \(\ZF\), even though the classical proofs (via idempotent ultrafilters on the partial semigroup \(\FIN_k\); see \cite{Todorcevic}) use the Axiom of Choice.  For each fixed \(k\) the statement is \(\Pi^1_2\): the set \(\FIN_k(\omega)\) is countable, so a finite coloring of it and a block sequence in it are coded by reals, and monochromaticity of the tetris sums is an arithmetic property of these codes.  Given a coloring \(c\), the inner model \(L[c]\) satisfies \(\ZFC\) and hence contains an infinite block sequence whose tetris sums are monochromatic for \(c\); Shoenfield absoluteness (see, e.g., \cite[Chapter~25]{JechST}) then transfers the existence of such a sequence to the universe.

\subsection{The associated finiteness class}

\begin{definition}
Let \(X\) be a set and \(k\geq 1\).  We say that \(X\) is \emph{\(G_k\)-finite} if there is a coloring
\[
    c:\FIN_k(X)\to 2
\]
such that no infinite block family \(Y\subseteq\FIN_k(X)\) has \(\TS(Y)\) monochromatic for \(c\).  The class of \(G_k\)-finite sets is denoted \(G_k\text{-}\Fin\).
\end{definition}

Equivalently, \(X\) is \(G_k\)-infinite if every two-coloring of \(\FIN_k(X)\) admits an infinite block family with monochromatic tetris sums.

\begin{proposition}
For each \(k\geq 1\), \(G_k\text{-}\Fin\) is a finiteness class.
\end{proposition}

\begin{proof}
Closure under subsets follows by pulling back a witnessing coloring along the embedding \(\FIN_k(Z)\hookrightarrow\FIN_k(X)\) given, for \(Z\subseteq X\), by extending functions by \(0\); this embedding preserves supports, block families, and tetris sums.  Closure under equipotence follows because a bijection \(\pi:X\to Y\) induces a bijection \(\pi_*:\FIN_k(X)\to\FIN_k(Y)\) preserving supports, block families, and tetris sums.  Finite sets are \(G_k\)-finite because no infinite block family can exist in \(\FIN_k(X)\) when \(X\) is finite.  Finally, \(\omega\notin G_k\text{-}\Fin\) by Gowers' theorem.
\end{proof}

\begin{proposition}
\(G_1\text{-}\Fin=\HFin\).
\end{proposition}

\begin{proof}
The identification \(\FIN_1(X)\cong [X]^{<\omega}\setminus\{\varnothing\}\) sends block families in \(\FIN_1(X)\) to pairwise disjoint families of nonempty finite subsets of \(X\).  Since the tetris operation is trivial at level \(1\), the tetris sums are exactly finite unions.  Therefore the definition of \(G_1\)-finiteness is precisely the definition of \(H\)-finiteness, up to the harmless omission of \(\varnothing\).
\end{proof}

\begin{proposition}\label{prop:gk-monotone}
For each \(k\geq 1\), \(G_k\text{-}\Fin\subseteq G_{k+1}\text{-}\Fin\).
\end{proposition}

\begin{proof}
Suppose \(X\) is \(G_k\)-finite, witnessed by \(c:\FIN_k(X)\to 2\).  Define, for \(f\in\FIN_{k+1}(X)\),
\[
    c'(f)=c(T(f)).
\]
We claim that \(c'\) witnesses \(G_{k+1}\)-finiteness.

Assume toward a contradiction that \(Y\subseteq\FIN_{k+1}(X)\) is an infinite block family such that \(\TS(Y)\) is monochromatic for \(c'\), say in color \(i\).  Put
\[
    T[Y]=\{T(y):y\in Y\}.
\]
Since \(T(y)\in\FIN_k(X)\), since \(\supp(T(y))\subseteq\supp(y)\), and since the supports of the elements of \(Y\) are pairwise disjoint, this is a block family in \(\FIN_k(X)\).  It is infinite because distinct elements of \(Y\) have disjoint nonempty supports, and applying \(T\) to an element of \(\FIN_{k+1}(X)\) cannot give the zero function.

We now show that \(\TS(T[Y])\) is monochromatic for \(c\).  Let
\[
    s=\sum_{j=1}^r T^{a_j}(T(y_j))
\]
be an arbitrary tetris sum from \(T[Y]\), where \(y_1,\dots,y_r\in Y\) are distinct and \(a_j\in\{0,\dots,k-1\}\).  Then
\[
    s=\sum_{j=1}^r T^{a_j+1}(y_j).
\]
Since \(s\in\FIN_k(X)\), at least one summand contributes a point at level \(k\).  For such a summand, necessarily \(a_j=0\) and \(y_j\) has a point at level \(k+1\).  Hence
\[
    s'=\sum_{j=1}^r T^{a_j}(y_j)
\]
belongs to \(\FIN_{k+1}(X)\), is a tetris sum from \(Y\), and satisfies \(T(s')=s\).  Therefore
\[
    c(s)=c(T(s'))=c'(s')=i.
\]
Thus \(\TS(T[Y])\) is monochromatic for \(c\), contradicting that \(c\) witnesses \(G_k\)-finiteness.
\end{proof}

The preceding proposition suggests a possible hierarchy
\[
    \HFin=G_1\text{-}\Fin\subseteq G_2\text{-}\Fin\subseteq G_3\text{-}\Fin\subseteq\cdots.
\]
The next theorem shows that this hierarchy collapses.

\begin{lemma}[{see\ \cite[Theorem~3.2]{BCF}}]\label{lem:disjoint-pieces}
Suppose \([X]^{<\omega}\) is Dedekind-infinite.  Then there is an infinite sequence \((P_j)_{j<\omega}\) of pairwise disjoint nonempty finite subsets of \(X\).
\end{lemma}

\begin{proof}
Since \([X]^{<\omega}\) is Dedekind-infinite, fix an injective sequence \(n\longmapsto F_n\) from \(\omega\) into \([X]^{<\omega}\).  For each \(n\), set \(P_n=F_n\setminus\bigcup_{k<n}F_k\).  We claim that \(P_n\neq\varnothing\) for infinitely many \(n\).  If not, there is some \(N\) such that \(P_n=\varnothing\) for all \(n\geq N\), i.e., \(F_n\subseteq\bigcup_{k<N}F_k\) for all \(n\geq N\).  But \(\bigcup_{k<N}F_k\) is finite, so its power set is finite, and the sequence \((F_n)_{n\geq N}\) takes values in a finite set, contradicting injectivity.  Enumerating the indices for which \(P_n\neq\varnothing\) gives the desired infinite sequence of pairwise disjoint nonempty finite subsets of \(X\).
\end{proof}

\begin{proposition}
Let \(k\geq 1\).  If \([X]^{<\omega}\) is Dedekind-infinite, then \(X\) is \(G_k\)-infinite.
\end{proposition}

\begin{proof}
By the preceding lemma, fix an infinite sequence \((P_j)_{j<\omega}\) of pairwise disjoint nonempty finite subsets of \(X\).

We define a map
\[
    \Phi:\FIN_k(\omega)\to \FIN_k(X)
\]
by spreading the value of a function on \(\omega\) across the corresponding finite block \(P_j\).  Namely, for \(f\in\FIN_k(\omega)\), define \(\Phi(f):X\to\{0,1,\dots,k\}\) by
\[
    \Phi(f)(x)=
    \begin{cases}
        f(j), & x\in P_j,\\
        0, & x\notin \bigcup_{j<\omega}P_j.
    \end{cases}
\]
This is well-defined because the \(P_j\)'s are pairwise disjoint.  Moreover, \(\Phi(f)\in\FIN_k(X)\): the support of \(\Phi(f)\) is contained in the finite union \(\bigcup_{j\in \supp(f)}P_j\), and since \(f\) attains the value \(k\) at some \(j\), the function \(\Phi(f)\) attains the value \(k\) on the nonempty set \(P_j\).

The map \(\Phi\) preserves the tetris operation and finite sums along disjoint supports.  More precisely,
\[
    \Phi(T^r f)=T^r\Phi(f)
\]
for each relevant \(r\), and if \(f_1,\dots,f_s\in\FIN_k(\omega)\) have pairwise disjoint supports, then
\[
    \Phi\left(\sum_{i=1}^s f_i\right)=\sum_{i=1}^s \Phi(f_i).
\]
Since each \(P_j\) is nonempty, \(\Phi\) is injective.  Consequently, \(\Phi\) sends infinite block sequences in \(\FIN_k(\omega)\) to infinite block families in \(\FIN_k(X)\), and it sends tetris sums to the corresponding tetris sums.

Now let
\[
    c:\FIN_k(X)\to 2
\]
be an arbitrary coloring.  Pull it back to a coloring
\[
    d:\FIN_k(\omega)\to 2
\]
by
\[
    d(f)=c(\Phi(f)).
\]
By Gowers' \(\FIN_k\) theorem on \(\omega\), there is an infinite block sequence \((f_i)_{i<\omega}\) in \(\FIN_k(\omega)\) such that \(\TS(\{f_i:i<\omega\})\) is monochromatic for \(d\).  Then \((\Phi(f_i))_{i<\omega}\) is an infinite block family in \(\FIN_k(X)\), and its tetris sums are monochromatic for \(c\).  Since \(c\) was arbitrary, \(X\) is \(G_k\)-infinite.
\end{proof}

\begin{theorem}\label{thm:gowers-collapse}
\begin{samepage}
For every set \(X\) and every \(k\geq 1\), the following are equivalent:
\begin{enumerate}[label=(\arabic*)]
    \item \(X\) is \(G_k\)-finite;
    \item \([X]^{<\omega}\) is Dedekind-finite;
    \item \(X\) is \(H\)-finite.
\end{enumerate}
\end{samepage}
Consequently, \(G_k\text{-}\Fin=\HFin\) for every \(k\geq 1\).
\end{theorem}

\begin{proof}
The equivalence of (2) and (3) is due to Brot, Cao, and Fern\'andez-Bret\'on \cite[Theorem~3.2]{BCF}.  Also, \(G_1\)-finiteness is exactly \(H\)-finiteness, and Proposition~\ref{prop:gk-monotone} gives
\[
    G_1\text{-}\Fin\subseteq G_k\text{-}\Fin.
\]
Thus (2) implies (1).

Conversely, suppose (2) fails.  Then \([X]^{<\omega}\) is Dedekind-infinite, so by the preceding proposition, \(X\) is \(G_k\)-infinite.  Hence (1) fails.  Therefore (1) implies (2), and all three conditions are equivalent.
\end{proof}

Theorem~\ref{thm:gowers-collapse} answers the question of Brot, Cao, and Fern\'andez-Bret\'on about Gowers' theorem: the classes \(G_k\text{-}\Fin\) do not form a proper hierarchy, since every one of them coincides with \(\HFin\).

We record one consequence of Theorem~\ref{thm:gowers-collapse} that will be used in Section~\ref{sec:independence} to certify \(H\)-finiteness in permutation models.

\begin{lemma}[{see\ \cite[Theorems~3.2 and~3.7]{BCF}}]\label{lem:amorphous-hfin}
Every amorphous set is \(H\)-finite.  More generally, if \(A\) is a set with no countably infinite pairwise disjoint family of nonempty finite subsets, then \(A\) is \(H\)-finite.
\end{lemma}

\begin{proof}
The substance of this lemma is already contained in \cite{BCF}.  For the second statement: if \([A]^{<\omega}\) were Dedekind-infinite, then Lemma~\ref{lem:disjoint-pieces} would produce a countably infinite pairwise disjoint family of nonempty finite subsets of \(A\); hence \([A]^{<\omega}\) is Dedekind-finite, and \(A\) is \(H\)-finite by Theorem~\ref{thm:gowers-collapse}.  (The same disjointification appears in the proof of \cite[Theorem~3.2]{BCF}.)  The first statement follows from \cite{BCF} directly: amorphous sets are \(C\)-finite, and every \(C\)-finite set is \(H\)-finite by \cite[Theorem~3.7]{BCF}.  Alternatively, an amorphous set \(A\) satisfies the hypothesis of the second statement, since for a countably infinite pairwise disjoint family \(\{P_j:j<\omega\}\) of nonempty finite subsets of \(A\), the sets \(\bigcup_{j\text{ even}}P_j\) and \(A\setminus\bigcup_{j\text{ even}}P_j\) would partition \(A\) into two infinite pieces.
\end{proof}

\section{Rainbow and canonical Ramsey finiteness}\label{sec:rainbow}

We now turn to Ramsey-theoretic principles that do produce new behavior relative to \(\HFin\).

\begin{definition}
Fix \(m,n\geq 2\).
\begin{enumerate}[label=(\arabic*)]
    \item A coloring \(\chi:[X]^m\to C\) is \emph{\(n\)-bounded} if \(|\chi^{-1}(c)|\leq n\) for every \(c\in C\).
    \item A set \(Y\subseteq X\) is \emph{polychromatic} for \(\chi\) if \(\chi\) is injective on \([Y]^m\).
    \item \(\RRT^m_n(X)\) is the assertion that every \(n\)-bounded coloring \(\chi:[X]^m\to C\) admits an infinite polychromatic \(Y\subseteq X\).
\end{enumerate}
\end{definition}

\begin{definition}\label{def:canonical-forms}
Let \(f:[X]^2\to C\) be a coloring and let \(<\) be a linear order on a subset \(Y\subseteq X\).  We say that \(f\) is \emph{canonical} on \((Y,<)\) if it has one of the following four Erd\H{o}s--Rado forms, which hold for all \(a_1,b_1,a_2,b_2\in Y\) with \(a_1<b_1\) and \(a_2<b_2\):
\begin{enumerate}[label=(\alph*)]
    \item \emph{constant:} \(f(\{a_1,b_1\})=f(\{a_2,b_2\})\);
    \item \emph{left-determined:} \(f(\{a_1,b_1\})=f(\{a_2,b_2\})\) if and only if \(a_1=a_2\);
    \item \emph{right-determined:} \(f(\{a_1,b_1\})=f(\{a_2,b_2\})\) if and only if \(b_1=b_2\);
    \item \emph{injective:} \(f(\{a_1,b_1\})=f(\{a_2,b_2\})\) if and only if \(a_1=a_2\) and \(b_1=b_2\).
\end{enumerate}
In particular, in a left-determined coloring the color of a pair is determined by its \(<\)-smaller element, and distinct \(<\)-smaller elements give distinct colors; symmetrically for the right-determined case.
\end{definition}

\begin{definition}\label{def:crt}
\(\CRT(X)\) is the assertion that for every coloring \(f:[X]^2\to C\), there exists an infinite set \(Y\subseteq X\) and a linear order \(<\) on \(Y\) such that \(f\) is canonical on \((Y,<)\) in the sense of Definition~\ref{def:canonical-forms}.
\end{definition}

We use the following facts from Banerjee, Gopaulsingh, and Moln\'ar \cite{BGM}.

\begin{proposition}[Banerjee--Gopaulsingh--Moln\'ar]\label{prop:bgm}
The following hold in \(\ZF\):
\begin{enumerate}[label=(\arabic*)]
    \item \(\RRT^m_n(X)\) and \(\CRT(X)\) hold for every well-orderable infinite set \(X\).
    \item If \(\RRT^m_n(Y)\) holds and \(Y\subseteq X\), then \(\RRT^m_n(X)\) holds.
    \item If \(\CRT(Y)\) holds and \(Y\subseteq X\), then \(\CRT(X)\) holds.
    \item \(\CRT(X)\Rightarrow \RRT^2_n(X)\) for every \(n\geq 2\).
    \item If \(k\leq n\), then \(\RRT^m_n(X)\Rightarrow \RRT^m_k(X)\).
\end{enumerate}
\end{proposition}

\begin{definition}
A set \(X\) is \emph{\(\RRT^m_n\)-finite} if \(\RRT^m_n(X)\) fails.  We write \(\RRTFin{m}{n}\) for the class of \(\RRT^m_n\)-finite sets.\footnote{In \cite{BGM}, the class of \(\RRT^m_n\)-finite sets is denoted \(\Gamma_{\RRT^m_n}\), and the class of \(\CRT\)-finite sets is denoted \(\Gamma_{\CRT}\).  We use the present notation for uniformity with the other finiteness classes.}

A set \(X\) is \emph{\(\CRT\)-finite} if \(\CRT(X)\) fails.  We write \(\CRTFin\) for the class of \(\CRT\)-finite sets.
\end{definition}

\begin{corollary}\label{cor:rrt-fin-class}
For every \(m,n\geq 2\), \(\RRTFin{m}{n}\) is a finiteness class.
\end{corollary}

\begin{proof}
Closure under subsets follows from the upward closure in Proposition~\ref{prop:bgm}(2), and closure under equipotence holds because a bijection \(X\to Y\) induces a bijection \([X]^m\to[Y]^m\) under which \(n\)-bounded colorings and polychromatic sets correspond.  Every finite set is trivially \(\RRT^m_n\)-finite, and \(\omega\notin\RRTFin{m}{n}\) by Proposition~\ref{prop:bgm}(1).
\end{proof}

\begin{corollary}\label{cor:crt-fin-class}
\(\CRTFin\) is a finiteness class.
\end{corollary}

\begin{proof}
The argument is analogous, using Proposition~\ref{prop:bgm}(1) and~(3).
\end{proof}

\section{Placement of \(\CRTFin\)}

\begin{proposition}\label{prop:bfin-crt}
\(\BFin\subseteq \CRTFin\).
\end{proposition}

\begin{proof}
If \(X\) is \(B\)-finite, then no infinite subset of \(X\) is linearly orderable.  But \(\CRT(X)\) requires an infinite subset \(Y\subseteq X\) equipped with a linear ordering.  Hence \(\CRT(X)\) fails.
\end{proof}

\begin{corollary}
Every amorphous set is \(\CRT\)-finite.
\end{corollary}

The next observation gives a useful link between canonical Ramsey finiteness and the Ramsey finiteness classes of \cite{BCF}.

\begin{theorem}\label{thm:crt-implies-rt}
For every finite \(n\geq 2\), \(\CRT(X)\Rightarrow \RT^2_n(X)\).  Consequently, \(\RnFin{n}\subseteq \CRTFin\).
\end{theorem}

\begin{proof}
Let \(c:[X]^2\to n\) be a finite coloring and assume \(\CRT(X)\).  Then there is an infinite \(Y\subseteq X\) and a linear order \(<\) on \(Y\) such that \(c\) is canonical on \((Y,<)\), in one of the four forms of Definition~\ref{def:canonical-forms}.  We claim that, because the color set \(n\) is finite while \(Y\) is infinite, the only possible form is the constant one.

Suppose \(c\) is left-determined on \((Y,<)\).  By Definition~\ref{def:canonical-forms}(b), the map sending each \(a\in Y\) to the common color of all pairs \(\{a,b\}\) with \(a<b\) is well-defined and \emph{injective}: distinct \(<\)-smaller elements yield distinct colors.  Hence this is an injection from the infinite set \(Y\) (minus at most its \(<\)-largest element, if one exists) into the finite set \(n\), which is impossible.  The right-determined case is symmetric, using Definition~\ref{def:canonical-forms}(c) and the \(<\)-smallest element.  The injective case is likewise impossible: by Definition~\ref{def:canonical-forms}(d) the coloring \(c\) would be injective on \([Y]^2\), so \([Y]^2\) would inject into the finite set \(n\) and hence be finite; but then \(Y\) would be finite as well, since every element of the infinite set \(Y\) occurs in some pair of \([Y]^2\), contradicting the infinitude of \(Y\).

Therefore \(c\) must be constant on \([Y]^2\), so \(Y\) is an infinite homogeneous set for \(c\), and \(\RT^2_n(X)\) holds.  Taking contrapositives gives \(\RnFin{n}\subseteq\CRTFin\).
\end{proof}

\begin{corollary}
For every \(n\geq 2\), \(\RRTFin{2}{n}\subseteq \CRTFin\).
\end{corollary}

\begin{remark}
The inclusion \(\BFin\subseteq\CRTFin\) is immediate from the definition, but the inclusion \(\RnFin{n}\subseteq\CRTFin\) uses the canonical alternatives in an essential way.  It says that canonical Ramsey homogeneity for pairs is strong enough to recover ordinary finite-color Ramsey homogeneity on the same set.
\end{remark}

\section{Independence from the Gowers and Hindman class}\label{sec:independence}

We now show that the rainbow and canonical finiteness classes do not collapse to \(\HFin\).  Since Theorem~\ref{thm:gowers-collapse} identifies \(\HFin\) with every \(G_k\)-finiteness class, these are also independence results from the entire Gowers hierarchy.

We state the results in terms of standard Fraenkel--Mostowski permutation models of \(\ZFA\); we follow the terminology and notation of Jech \cite[Chapter~4]{JechAC}.  In each such model, there is a set \(A\) of atoms, a group \(\mathcal G\) of permutations of \(A\), and a normal filter \(\mathcal F\) of subgroups of \(\mathcal G\) generated by pointwise stabilizers of finite sets (called \emph{supports}).  A set is \emph{symmetric} if its pointwise stabilizer belongs to \(\mathcal F\).  For a finite set \(E\subseteq A\), we write \(\mathrm{fix}(E)=\{g\in\mathcal G\mid g(a)=a\text{ for all }a\in E\}\).  We use three standard models: the basic Fraenkel model \(\NN_1\), the second Fraenkel model \(\NN_2\), and Mostowski's linearly ordered model \(\NN_3\); see \cite[Chapter~4]{JechAC} for their construction.  The labels \(\NN_1\), \(\NN_2\), and \(\NN_3\) follow the catalogue of Howard and Rubin \cite{HowardRubin}.

Each independence below is established as a class non-inclusion witnessed by the atom set of a permutation model.  The relevant statements---the truth or failure of \(\RRT^m_n(A)\) and of \(\CRT(A)\), and Dedekind-finiteness or Dedekind-infiniteness of \([A]^{<\omega}\), to which membership in \(\HFin\) reduces by Theorem~\ref{thm:gowers-collapse}---concern colorings of \([A]^m\), subsets of \(A\), and linear orders on subsets of \(A\).  In the statements that quantify over all colorings we may restrict attention to colorings whose values are their own fibers: replacing \(\chi\) by \(x\mapsto\chi^{-1}(\chi(x))\) yields a coloring into \(\mathcal P([A]^m)\) with the same fibers, hence with the same \(n\)-boundedness, the same polychromatic sets, and the same canonical forms.  With this normalization every statement involved is a bounded assertion about objects of low rank over \(A\).  Consequently, by the Jech--Sochor first embedding theorem \cite[Theorem~6.1]{JechAC}, for each such model there is a symmetric extension of a model of \(\ZF\) containing a set realizing the same configuration of finiteness properties; the corresponding inclusions are therefore not provable in \(\ZF\).  We work in \(\ZFA\) in this section and in Section~\ref{sec:parameter}, and leave the routine transfer implicit.

\begin{theorem}\label{thm:rrt-hfin-indep}
For every \(m,n\geq 2\), the classes \(\RRTFin{m}{n}\) and \(\HFin\) are independent.
\end{theorem}

\begin{proof}
First, \(\HFin\not\subseteq\RRTFin{m}{n}\).  In the basic Fraenkel model \(\NN_1\), the set \(A\) of atoms is amorphous \cite[Model \(\mathcal N1\)]{HowardRubin}, hence \(H\)-finite by Lemma~\ref{lem:amorphous-hfin}.  On the other hand, Banerjee, Gopaulsingh, and Moln\'ar \cite[Theorem~4.1(1)]{BGM} show that \(\RRT^m_n(A)\) holds in \(\NN_1\) for all \(m,n\geq 2\).  Thus
\[
    A\in\HFin\setminus\RRTFin{m}{n}.
\]

Conversely, \(\RRTFin{m}{n}\not\subseteq\HFin\).  In the second Fraenkel model \(\NN_2\), the set \(A\) of atoms is partitioned into a countable family \(\{P_k:k<\omega\}\) of pairs admitting no infinite partial choice function; that is, the principle \(\mathrm{AC}^-_2\) (every infinite family of \(2\)-element sets has an infinite subfamily with a choice function) fails for this family in \(\NN_2\) \cite[Model \(\mathcal N2\)]{HowardRubin}.  We use this to show \(\RRT^m_n(A)\) fails for all \(m,n\geq 2\).

Banerjee, Gopaulsingh, and Moln\'ar \cite[Proposition~3.3(1)]{BGM} prove in \(\ZF\) that \(\RRT^m_2\) implies \(\mathrm{AC}^-_2\), by associating to any family \(\mathcal A\) of \(2\)-element sets an explicit \(2\)-bounded coloring of \([\bigcup\mathcal A]^m\) for which every infinite polychromatic set yields an infinite partial choice function on \(\mathcal A\).  Applying this construction to the family \(\{P_k:k<\omega\}\) inside \(\NN_2\) produces a \(2\)-bounded coloring \(\chi:[A]^m\to C\) (with \(A=\bigcup_k P_k\)) such that any infinite polychromatic \(B\subseteq A\) would determine an infinite partial choice function for \(\{P_k:k<\omega\}\).  Since no such choice function exists in \(\NN_2\), the coloring \(\chi\) has no infinite polychromatic set, so \(\RRT^m_2(A)\) fails.  As \(\chi\) is \(2\)-bounded it is in particular \(n\)-bounded for every \(n\geq 2\), and it still admits no infinite polychromatic set; equivalently, by Proposition~\ref{prop:bgm}(5) the failure of \(\RRT^m_2(A)\) propagates upward to the failure of \(\RRT^m_n(A)\) for every \(n\geq 2\).  Hence \(\RRT^m_n(A)\) fails for all \(m,n\geq 2\), and \(A\in\RRTFin{m}{n}\).

(For the case \(m=2\), this recovers Palumbo's theorem that \(\RRT^2_2\) fails in \(\NN_2\) \cite[Theorem~2.3]{Palumbo}.)

Finally, the pair decomposition \(\{P_k:k<\omega\}\) is itself a countably infinite family of distinct nonempty finite subsets of \(A\), so it witnesses that \([A]^{<\omega}\) is Dedekind-infinite.  By the Gowers collapse theorem (Theorem~\ref{thm:gowers-collapse}), \(A\notin\HFin\).
\end{proof}

\begin{lemma}[{\cite[Proposition~4.17]{BCF}}]\label{lem:n3-hfin}
In the Mostowski linearly ordered model \(\NN_3\), the set \(A\) of atoms is \(H\)-finite.
\end{lemma}

\begin{proof}
Brot, Cao, and Fern\'andez-Bret\'on prove the stronger statement that \(A\) is \(H_3\)-finite \cite[Proposition~4.17]{BCF}; since \(H_3\)-finiteness implies \(H\)-finiteness (see \cite[Section~3]{BCF}), the lemma follows.
\end{proof}

\begin{theorem}\label{thm:crt-hfin-indep}
The classes \(\CRTFin\) and \(\HFin\) are independent.
\end{theorem}

\begin{proof}
For \(\HFin\not\subseteq\CRTFin\), use the Mostowski linearly ordered model \(\NN_3\).  In that model, the set \(A\) of atoms is \(H\)-finite by Lemma~\ref{lem:n3-hfin}, while \(\CRT(A)\) holds by Banerjee, Gopaulsingh, and Moln\'ar \cite[Theorem~4.8(2)]{BGM}.  Hence \(A\in\HFin\setminus\CRTFin\).

For \(\CRTFin\not\subseteq\HFin\), use the second Fraenkel model \(\NN_2\).  The set \(A\) of atoms is a Russell set, and every Russell set is \(B\)-finite: an infinite linearly orderable subset would single out one atom from each of infinitely many pairs, yielding an infinite partial choice function on the pairs, which \(\NN_2\) does not admit.  By Proposition~\ref{prop:bfin-crt}, \(B\)-finiteness gives \(A\in\CRTFin\).  As in the proof of the previous theorem, the pair decomposition makes \([A]^{<\omega}\) Dedekind-infinite, so \(A\notin\HFin\) by Theorem~\ref{thm:gowers-collapse}.
\end{proof}

\section{Rainbow parameter behavior}\label{sec:parameter}

For fixed \(m\), the rainbow classes are monotone in the bound parameter.

\begin{proposition}\label{prop:rrt-n-monotone}
For fixed \(m\geq 2\),
\[
    \RRTFin{m}{2}
    \subseteq
    \RRTFin{m}{3}
    \subseteq
    \RRTFin{m}{4}
    \subseteq\cdots.
\]
\end{proposition}

\begin{proof}
If \(k\leq n\), then every \(k\)-bounded coloring is \(n\)-bounded.  Thus a \(k\)-bounded coloring witnessing failure of \(\RRT^m_k(X)\) is also an \(n\)-bounded coloring witnessing failure of \(\RRT^m_n(X)\).
\end{proof}

We now record two mechanisms that will be used to understand the parameter behavior.  The first is positive: in Fraenkel--Mostowski models where every atom outside a finite parameter set has infinite orbit over that set (a property we call \emph{trivial algebraic closure}; see Definition~\ref{def:tac} below), every bounded coloring is forced to be injective away from the support, so the rainbow Ramsey theorem holds.  The second is negative: in models whose atoms are partitioned into uniform finite blocks (see Definition~\ref{def:block-model} below), one can build bounded colorings by pairing each finite configuration within a block with a complementary one.

\subsection{A support-orbit criterion for rainbow Ramsey}

The next lemma is the basic technical point behind the positive results.  It is useful to separate it from the algebraic closure argument that follows, since this is exactly where the boundedness hypothesis enters.

\begin{lemma}[Support-orbit lemma]
Let \(\NN\) be a Fraenkel--Mostowski model with atom set \(A\), and let \(\chi:[A]^m\to C\) be a \(k\)-bounded coloring with finite support \(E\).  Suppose \(x,y\in[A]^m\) satisfy \(x\cap E=y\cap E=\varnothing\) and \(\chi(x)=\chi(y)\).  Then \(y\) has finite orbit under \(\mathrm{fix}(E\cup x)\).  In fact,
\[
    \left|\mathrm{Orb}_{\mathrm{fix}(E\cup x)}(y)\right|\leq k.
\]
\end{lemma}

\begin{proof}
Let \(g\in\mathrm{fix}(E\cup x)\).  Since \(E\) supports \(\chi\), we have \(\chi(gz)=g(\chi(z))\) for every \(z\in[A]^m\).  Since \(g\) fixes \(x\), it follows that \(g(\chi(x))=\chi(gx)=\chi(x)\).  Now \(\chi(y)=\chi(x)\), so
\[
    \chi(gy)=g(\chi(y))=g(\chi(x))=\chi(x).
\]
Thus every translate \(gy\) lies in the fiber \(\chi^{-1}(\chi(x))\).  This fiber has size at most \(k\), so the orbit of \(y\) under \(\mathrm{fix}(E\cup x)\) has size at most \(k\).
\end{proof}

Given a Fraenkel--Mostowski model with atoms \(A\) and a finite set \(F\subseteq A\), the \emph{algebraic closure} of \(F\) is defined by
\[
    \acl(F)=\{a\in A\mid \text{the orbit of }a\text{ under }\mathrm{fix}(F)\text{ is finite}\}.
\]

\begin{definition}\label{def:tac}
A Fraenkel--Mostowski model with atoms \(A\) has \emph{trivial algebraic closure} if \(\acl(F)=F\) for every finite \(F\subseteq A\).
\end{definition}

\begin{theorem}\label{thm:tac}
Let \(\NN\) be a Fraenkel--Mostowski model whose atoms have trivial algebraic closure.  Then \(\RRT^m_k(A)\) holds for all \(m,k\geq 2\).
\end{theorem}

\begin{proof}
Let \(\chi:[A]^m\to C\) be a \(k\)-bounded coloring, and let \(E\) be a finite support for \(\chi\).  We claim that \(\chi\) is injective on \([A\setminus E]^m\).

Suppose \(x,y\in[A\setminus E]^m\) and \(\chi(x)=\chi(y)\).  By the support-orbit lemma, \(y\) has finite orbit under \(\mathrm{fix}(E\cup x)\).  Hence every atom \(a\in y\) also has finite orbit under \(\mathrm{fix}(E\cup x)\): indeed, the orbit of \(a\) is contained in the finite union
\[
    \bigcup\{z\mid z\in \mathrm{Orb}_{\mathrm{fix}(E\cup x)}(y)\}.
\]
Therefore
\[
    a\in\acl(E\cup x)=E\cup x.
\]
Since \(y\cap E=\varnothing\), this gives \(y\subseteq x\).  But \(|x|=|y|=m\), so \(x=y\).  Thus \(\chi\) is injective on \([A\setminus E]^m\), and \(A\setminus E\) is an infinite polychromatic subset of \(A\).
\end{proof}

\begin{corollary}
Let \(\NN\) be a Fraenkel--Mostowski model with trivial algebraic closure whose atoms \(A\) are the vertices of a countably infinite homogeneous relational structure,\footnote{Such as the Rado graph, a Rado \(r\)-uniform hypergraph, or a linearly ordered variant thereof.  Homogeneity alone does not suffice here: a countable perfect matching is a homogeneous graph in which every vertex is algebraic over its partner.} with the automorphism group of the structure acting on \(A\) and finite supports.  Then \(\RRT^m_k(A)\) holds for all \(m,k\geq 2\).
\end{corollary}

\begin{proof}
The result is immediate from Theorem~\ref{thm:tac}.  For the structures mentioned in the footnote, the algebraic closure is indeed trivial: over a finite set of parameters, every atom outside the parameter set has infinitely many realizations of the same quantifier-free type, so its orbit under the pointwise stabilizer is infinite.
\end{proof}

\subsection{The three-block model}

We next analyze the simplest block model that separates the first two levels of the \(n\)-parameter chain.  The argument is included in some detail because it is easy to wave one's hands here and accidentally use more homogeneity than the model actually has.

\begin{definition}\label{def:block-model}
The \emph{\(3\)-block model} is the Fraenkel--Mostowski model in which
\[
    A=\bigcup_{i<\omega}B_i,
    \qquad |B_i|=3,
\]
the group \(\mathcal G\) consists of all permutations \(\pi:A\to A\) satisfying \(\pi[B_i]=B_i\) for every \(i<\omega\), and supports are finite.
\end{definition}

We first record the following support fact.

\begin{lemma}\label{lem:block-support}
Every symmetric subset \(Y\subseteq A\) in the \(3\)-block model is the union of whole blocks and a finite set.  More precisely, if \(E\) is a finite support for \(Y\), then \(Y\cap B_i\in\{\varnothing,B_i\}\) for every block \(B_i\) disjoint from \(E\).  Consequently, if \(Y\) is infinite, then \(Y\) contains infinitely many whole blocks disjoint from \(E\).
\end{lemma}

\begin{proof}
If \(B_i\cap E=\varnothing\), then every permutation of \(B_i\) fixes \(E\) pointwise and hence preserves \(Y\).  Since the symmetric group \(\mathrm{Sym}(B_i)\) acts transitively on \(B_i\), either no element of \(B_i\) lies in \(Y\), or every element of \(B_i\) lies in \(Y\).  Since \(Y\) is infinite and each block is finite, this must happen for infinitely many blocks.
\end{proof}

\begin{proposition}\label{prop:block-rrt23-fails}
In the \(3\)-block model, \(\RRT^2_3(A)\) fails.
\end{proposition}

\begin{proof}
Define a coloring \(c:[A]^2\to C\) by
\[
    c(\{a,b\})=
    \begin{cases}
        i, & \text{if }a,b\in B_i,\\
        \{a,b\}, & \text{if }a,b\text{ lie in distinct blocks.}
    \end{cases}
\]
This coloring is \(3\)-bounded: each color \(i<\omega\) is used exactly three times, once for each pair in \([B_i]^2\), and every cross-block pair receives its own color.  It is also symmetric with empty support.

Let \(Y\subseteq A\) be infinite.  Let \(E\) be a finite support for \(Y\).  By Lemma~\ref{lem:block-support}, \(Y\) contains some whole block \(B_i\) disjoint from \(E\).  Then all three pairs in \([B_i]^2\) lie in \([Y]^2\) and have the same color.  Thus \(Y\) is not polychromatic.  Therefore \(\RRT^2_3(A)\) fails.
\end{proof}

It remains to show that \(\RRT^2_2(A)\) holds in the same model.  The key observation is the following: given a \(2\)-bounded coloring, equality of colors defines an invariant equivalence relation on pairs with classes of size at most \(2\), and we show that the \(3\)-block action admits no such nontrivial equivalence relation on pairs away from the support.

For a finite set \(E\subseteq A\), write \(A_E=\bigcup\{B_i\mid B_i\cap E=\varnothing\}\) for the union of all blocks disjoint from \(E\).

\begin{lemma}
Suppose \(\sim\) is a \(\mathrm{fix}(E)\)-invariant equivalence relation on \([A_E]^2\) whose equivalence classes have size at most \(2\).  Then \(\sim\) is equality.
\end{lemma}

\begin{proof}
The action of \(\mathrm{fix}(E)\) partitions \([A_E]^2\) into orbits of two kinds.  For each block \(B_i\) disjoint from \(E\), the set \([B_i]^2\) of internal pairs forms one orbit (of size \(3\)).  For each pair of distinct blocks \(B_i,B_j\) disjoint from \(E\), the set \(\{\{a,b\}\mid a\in B_i,\, b\in B_j\}\) of cross-block pairs forms one orbit (of size \(9\)).  We show that \(\sim\) must be equality on each orbit.

On \([B_i]^2\), the group \(\mathrm{Sym}(B_i)\) acts transitively on three points.  Therefore an invariant equivalence relation on this orbit has classes of size either \(1\) or \(3\).  Since the classes have size at most \(2\), only equality is possible.

Now consider the cross-block orbit between \(B_i\) and \(B_j\).  Fix \(p=\{a,b\}\) with \(a\in B_i\) and \(b\in B_j\).  The stabilizer of \(p\) in \(\mathrm{Sym}(B_i)\times \mathrm{Sym}(B_j)\) fixes \(p\), has one orbit of size \(2\) on the pairs \(\{a,b'\}\) with \(b'\in B_j\setminus\{b\}\), one orbit of size \(2\) on the pairs \(\{a',b\}\) with \(a'\in B_i\setminus\{a\}\), and one orbit of size \(4\) on the remaining pairs \(\{a',b'\}\).  Thus any invariant equivalence class containing \(p\) is either \(\{p\}\), or has size at least \(3\).  Since all classes have size at most \(2\), the class of \(p\) is \(\{p\}\).  As \(p\) was arbitrary, \(\sim\) is equality on this orbit.

It remains to rule out equivalence between two different orbits.  Suppose \(p\sim q\), with \(p\) and \(q\) lying in distinct orbits and \(p\neq q\).  If some block meeting \(q\) does not meet \(p\), then the stabilizer of \(p\) contains the full symmetric group on that block.  Since \(q\) contains at most two of the three atoms in that block, some permutation fixing \(p\) moves \(q\) to a different pair \(q'\).  Invariance gives \(p\sim q'\), so the class of \(p\) contains \(p,q,q'\), a contradiction.

Thus every block meeting \(q\) also meets \(p\).  By symmetry, every block meeting \(p\) also meets \(q\).  Hence \(p\) and \(q\) meet the same set of blocks.  But if a pair meets one block, it is an internal pair, and all internal pairs in that block form one orbit.  If a pair meets two blocks, it is a cross-block pair, and all cross-block pairs between those two blocks form one orbit.  This contradicts the assumption that \(p\) and \(q\) lie in distinct orbits.  Therefore no such \(p,q\) exist, and \(\sim\) is equality on all of \([A_E]^2\).
\end{proof}

\begin{theorem}\label{thm:block-rrt22-holds}
In the \(3\)-block model, \(\RRT^2_2(A)\) holds.
\end{theorem}

\begin{proof}
Let \(\chi:[A]^2\to C\) be a \(2\)-bounded coloring, and let \(E\) be a finite support for \(\chi\).  The set \(A_E\) is infinite.

Define an equivalence relation on \([A_E]^2\) by
\[
    p\sim q
    \quad\Longleftrightarrow\quad
    \chi(p)=\chi(q).
\]
Since \(E\) supports \(\chi\), this equivalence relation is \(\mathrm{fix}(E)\)-invariant.  Since \(\chi\) is \(2\)-bounded, every equivalence class has size at most \(2\).  By the previous lemma, \(\sim\) is equality.  Hence \(\chi\) is injective on \([A_E]^2\), and \(A_E\) is an infinite polychromatic subset of \(A\).  Thus \(\RRT^2_2(A)\) holds.
\end{proof}

\begin{corollary}
The inclusion \(\RRTFin{2}{2}\subseteq \RRTFin{2}{3}\) is consistently strict: in the \(3\)-block model, \(A\in \RRTFin{2}{3}\setminus \RRTFin{2}{2}\).  In particular, \(\ZF\) does not prove \(\RRTFin{2}{2}=\RRTFin{2}{3}\).  (Strictness is of course not provable outright either, since under \(\AC\) every finiteness class coincides with \(\Fin\).)
\end{corollary}

\begin{proof}
The inclusion is Proposition~\ref{prop:rrt-n-monotone}.  By Theorem~\ref{thm:block-rrt22-holds} and Proposition~\ref{prop:block-rrt23-fails}, in the \(3\)-block model \(\RRT^2_2(A)\) holds while \(\RRT^2_3(A)\) fails, so \(A\in \RRTFin{2}{3}\setminus \RRTFin{2}{2}\); the transfer to \(\ZF\) is as described in Section~\ref{sec:independence}.
\end{proof}

\subsection{A first separation in the arity parameter}

The same \(3\)-block model also gives an obstruction at arity \(3\).  This does not by itself prove a general monotone hierarchy in the \(m\)-direction: no implication between \(\RRT^2_2\) and \(\RRT^3_2\) is available in \(\ZF\) in either direction---the model refutes \(\RRT^2_2\Rightarrow\RRT^3_2\), and the converse is not known to be provable (see Question~\ref{q:m-direction}).  What it does show is that the arity parameter is not merely cosmetic: the same set can satisfy \(\RRT^2_2\) while failing \(\RRT^3_2\).

\begin{theorem}\label{thm:block-rrt32-fails}
In the \(3\)-block model, \(\RRT^3_2(A)\) fails.
\end{theorem}

\begin{proof}
\begin{samepage}
We define a \(2\)-bounded coloring \(c:[A]^3\to C\).  Call a triple \(T\in[A]^3\) \emph{bicommunal} if it meets exactly two blocks, with intersection sizes \(2\) and \(1\).  Thus
\[
    T=\{a,b,x\},
\]
where \(a,b\in B_i\), \(x\in B_j\), and \(i\neq j\).  Define the block-complement of \(T\) by
\[
    T^*=(B_i\setminus\{a,b\})\cup (B_j\setminus\{x\}).
\]
Then \(T^*\) is again a bicommunal triple, \(T^*\neq T\), and \((T^*)^*=T\).
\end{samepage}

Set \(c(T)=\{T,T^*\}\) for bicommunal triples \(T\), and \(c(S)=\{S\}\) for all other triples \(S\).  This coloring is \(2\)-bounded.  It is also symmetric with empty support, since every permutation preserving the block partition sends block-complements to block-complements.

Let \(Y\subseteq A\) be infinite, and let \(E\) be a finite support for \(Y\).  By Lemma~\ref{lem:block-support}, there are distinct blocks \(B_i,B_j\) disjoint from \(E\) such that
\[
    B_i\cup B_j\subseteq Y.
\]
Choose any bicommunal triple \(T\subseteq B_i\cup B_j\).  Then \(T^*\subseteq B_i\cup B_j\) as well, so \(T,T^*\in[Y]^3\).  But \(T\neq T^*\) and \(c(T)=c(T^*)\).

Thus \(Y\) is not polychromatic.  Since \(Y\) was arbitrary, there is no infinite polychromatic subset of \(A\), and \(\RRT^3_2(A)\) fails.
\end{proof}

\begin{corollary}
The classes \(\RRTFin{2}{2}\) and \(\RRTFin{3}{2}\) are consistently distinct: in the \(3\)-block model, \(A\in \RRTFin{3}{2}\setminus \RRTFin{2}{2}\).  In particular, \(\ZF\) does not prove \(\RRTFin{3}{2}\subseteq \RRTFin{2}{2}\).
\end{corollary}

\begin{proof}
This follows from Theorem~\ref{thm:block-rrt22-holds} and Theorem~\ref{thm:block-rrt32-fails}, together with the transfer described in Section~\ref{sec:independence}.
\end{proof}

\begin{remark}
The block-complement construction generalizes formally: in a \(2r\)-block model, an \(r\)-element subset of a block can be paired with its complement, giving a \(2\)-bounded coloring that defeats \(\RRT^r_2\).  What is not automatic is the positive half needed to turn this into a full strict chain in the arity parameter.  One would need to prove, for example, that the relevant lower-arity bounded colorings are forced to be injective off a finite support.  This reduces to a finite orbit-equivalence calculation for products of symmetric groups.  We have therefore left the general strict hierarchy as an open problem rather than building it into the main theorem statement.
\end{remark}

\begin{question}
For fixed \(m\geq 2\), is it consistent with \(\ZF\) that the chain
\[
    \RRTFin{m}{2}
    \subseteq
    \RRTFin{m}{3}
    \subseteq
    \RRTFin{m}{4}
    \subseteq\cdots
\]
is strict at every level?  The \(3\)-block model shows consistency of the first strict inclusion when \(m=2\).
\end{question}

\begin{question}\label{q:m-direction}
For fixed \(n\geq 2\), what is the exact relationship in \(\ZF\) between the classes
\[
    \RRTFin{m}{n}\qquad (m\geq 2)?
\]
The \(3\)-block model shows that arity matters, but it does not by itself produce a monotone hierarchy in \(m\).
\end{question}

\section{Summary and open questions}

The Gowers hierarchy collapses completely:
\[
    G_k\text{-}\Fin=\HFin=\{X:[X]^{<\omega}\text{ is Dedekind-finite}\}
\]
for every \(k\geq 1\).  The rainbow and canonical Ramsey classes behave differently.  They form genuine finiteness classes inside \(\DFin\), but they are independent from \(\HFin\), and therefore independent from the entire collapsed Gowers hierarchy.

One schematic way of summarizing the inclusions proved above is the following:
\[
\begin{array}{c}
\DFin \\[0.6em]
\cup \\[0.2em]
\CRTFin \\[0.6em]
\cup \\[0.2em]
\BFin,\quad \RnFin{n},\quad \RRTFin{2}{n}
\end{array}
\]
Here the bottom row indicates three classes which each embed into \(\CRTFin\), not three classes whose mutual relationships are being asserted.  Separately,
\[
    \Fin\subseteq \HFin=G_k\text{-}\Fin\subseteq \DFin,
\]
and \(\HFin\) is independent from both \(\CRTFin\) and \(\RRTFin{m}{n}\) in the sense proved above.
We close with the main remaining questions.

\begin{question}
Determine the exact relationship in \(\ZF\) between \(\CRTFin\) and the Ramsey finiteness classes \(\RnFin{n}\).  In particular, for which \(n\) is it consistent that the inclusion \(\RnFin{n}\subseteq\CRTFin\) is strict?
\end{question}

\begin{question}
Determine the full two-parameter structure of the rainbow classes \(\RRTFin{m}{n}\).  Are the monotone inclusions in the \(n\)-direction always strict?  Which inclusions in the \(m\)-direction are provable or refutable in \(\ZF\)?
\end{question}

\begin{question}
Can the block-complement obstruction be generalized to a complete family of Fraenkel--Mostowski models witnessing strictness in the \(m\)-direction?
\end{question}

\begin{question}
Are there natural restricted rainbow finiteness notions, analogous to the intermediate \(H_2\), \(H_3\), \(H\) hierarchy of \cite{BCF}, that fail to collapse in the way the unrestricted Gowers classes collapse?
\end{question}

\section*{Acknowledgements}

This paper forms part of the author's Ph.D.\ thesis at the University of Toronto.  The author thanks her supervisors, Spencer Unger and Asaf Karagila, for their guidance and support, and thanks Asaf Karagila in particular for his detailed comments on earlier drafts of this paper.

\bibliographystyle{amsplain}
\bibliography{mybib}

\end{document}